  \title{Random graphs with few disjoint cycles}
  \date{20 September 2009}

  \documentclass[11pt]{article}
  \usepackage{amsmath}
  \usepackage{amssymb}
  \usepackage{latexsym}
  \usepackage{graphicx}

  \newenvironment{proof}{\noindent{\bf Proof}}{\hspace*{\fill}$\Box$}
  \newenvironment{proofof}[1]{%
  \noindent {\bf Proof of #1} \hspace{.05in}}%
  {\hspace*{\fill}$\Box$}

  \newtheorem{theorem}{Theorem}[section]
  \newtheorem{lemma} [theorem] {Lemma}%[section]
\let\eps=\epsilon
\def\enddiscard{}
\long\def\discard#1\enddiscard{}

  \newcommand{\pr}{\mathbb P}
  
  \newcommand{\Bigc}{\rm Big}

  \newcommand{\Frag}{\rm Frag}

  \newcommand{\inu}{\in_{u}}

  \newcommand{\ca}{{\mathcal A}}
  \newcommand{\cb}{{\mathcal B}}
  \newcommand{\cc}{{\mathcal C}}
  \newcommand{\cd}{{\mathcal D}}
  \newcommand{\cf}{{\mathcal F}}
  
  \newcommand{\ch}{{\mathcal H}}
  \newcommand{\ck}{{\mathcal K}}
  \newcommand{\cs}{{\mathcal S}}
  \newcommand{\ct}{{\mathcal T}}

  \newcommand{\cw}{{\mathcal W}}

  \newcommand{\core}{{\rm core}}
  \newcommand{\ex}{{\rm Ex}}
  
  \newcommand{\apex}{\mbox{\rm apex \!}}

  \newcommand{\bin}{\mbox{Bin}}

\begin{document}

  \author{Valentas Kurauskas and Colin McDiarmid\\
  University of Oxford}
  \maketitle

  \begin{abstract}
  The classical Erd\H{o}s-P\'{o}sa theorem states that
  for each positive integer $k$ there is an $f(k)$
  such that, in each graph $G$ which does not have $k+1$ disjoint cycles,
  there is a blocker of size at most $f(k)$; that is,
  a set $B$ of at most $f(k)$ vertices such that $G-B$ has no cycles.
  We show that, amongst all such graphs on vertex set $\{1,\ldots,n\}$,
  all but an exponentially small proportion have a blocker of size $k$.
  %A corresponding result holds if for example we consider only cycles of
  %length at least some given bound.
  We also give further properties of a random graph sampled
  uniformly from this class; concerning uniqueness of the blocker, connectivity,
  chromatic number and clique number.

  A key step in the proof of the main theorem is to show that there must be a blocker
  as in the Erd\H{o}s-P\'{o}sa theorem with the extra `redundancy' property that
  $B-v$ is still a blocker for all but at most $k$ vertices $v \in B$.
  \end{abstract}
  \bigskip
  \bigskip

  keywords: disjoint cycles, blocker, Erd\H{o}s-P\'{o}sa theorem, apex graph, growth constant, random graph
  \bigskip

  corresponding author:

  Colin McDiarmid, Department of Statistics, 1 South Parks Road, Oxford OX1 3TG, UK

  tel: +44 1865 272860  \hspace{.1in} fax: +44 1865 272595
  \bigskip

  email:

  Valentas Kurauskas: valentas@gmail.com
  \smallskip

  Colin McDiarmid:  cmcd@stats.ox.ac.uk
 \newpage

\section{Introduction} \label{sec.intro}

  Call a set $B$ of vertices in a graph $G$ a {\em blocker} if the graph $G-B$
  obtained by deleting the vertices in $B$ has no cycles.
  The classical theorem of Erd\H{o}s and P\'{o}sa~\cite{ep65} from 1965
  (see for example~\cite{diestel})
  states that for each positive integer $k$ there is a positive integer $f(k)$
  such that the following holds:
  for each graph $G$ which does not have $k+1$ disjoint cycles
  (that is, pairwise vertex-disjoint cycles),
  there is a blocker $B$ of size at most $f(k)$.
  The least value we may take for $f(k)$ is of order $k \ln k$.

  We let $\cf$ denote the class of forests; let $\apex^k \cf$ denote the class of
  graphs with a blocker of size at most $k$; and let $\ex(k+1) C$
  denote the class of graphs which do not have $k+1$ disjoint cycles.
  With this notation the Erd\H{o}s-P\'{o}sa theorem says that
\[
  \ex(k+1) C \subseteq \apex^{f(k)} \cf.
\]
  Now clearly
\begin{equation} \label{eqn.forestapex}
  \ex (k+1) C \supseteq \apex^k \cf;
\end{equation}
  for if a graph has a blocker $B$ then
  %each cycle must contain a vertex in $B$, and so $G$
  it can have at most $|B|$ disjoint cycles.
  How much bigger is the left side of~(\ref{eqn.forestapex}) than the right?
  Our main theorem is that the difference is relatively small:
  amongst all graphs without $k+1$ disjoint cycles,
  all but a small proportion have a blocker of size $k$.
  For any class $\ca$ of graphs we let $\ca_n$ denote the set of graphs in $\ca$
  on the vertex set $\{1,\ldots,n\}$.
 (When we say a `class' of graphs it is assumed to be closed under automorphism.)

\begin{theorem} \label{thm.cycles}
  For each fixed positive integer $k$,  as $n \to \infty$
\begin{equation} \label{eqn.thmcycles}
  |(\ex (k\!+\!1) C )_n| = (1 + e^{-\Omega(n)})\ |(\apex^{k} \cf)_n|.
\end{equation}
\end{theorem}

  Graphs in $\ex \, 2C$ (that is, with no two disjoint cycles) have been well
  characterised (Dirac~\cite{dirac65}, Lov\'{a}sz~\cite{lovasz65});
  and from this characterisation we can much refine the above result for the case $k=1$
  -- see Section~\ref{sec.no2} below.
  %work the case $k=1$ of the above theorem can be deduced
  %(though we do not follow that route).
  It seems that no such characterisation is known for graphs in $\ex\ 3C$.

  The natural partner to Theorem~\ref{thm.cycles} is an asymptotic estimate
  for $|(\apex^{k}\cf)_n|$. Recall the result of R\'enyi (1959)~\cite{renyi59}
  that %, as $n \to \infty$
  \begin{equation} \label{eqn.renyi}
    |\cf_n| \sim e^{\frac12} n^{n-2} \sim
    \left(\frac{e}{2\pi}\right)^{\frac12} n^{-\frac52} e^n n!
    \hspace{.3in} \mbox{ as } n \to \infty.
  \end{equation}

\begin{theorem} \label{thm.apexforest}
  For each fixed positive integer $k$, as $n \to \infty$
\begin{equation} \label{eqn.apexforest}
  |(\apex^{k} \cf)_n| \sim  c_k 2^{kn} |\cf_n|
\end{equation}
  where $c_k = \left( 2^{\binom {k+1} 2} e^{k} k! \right)^{-1}$.
\end{theorem}
  A class $\ca$ of graphs has {\em growth constant} $\gamma$ if
\[
  \left(|\ca_n| / n! \right)^{1/n} \to \gamma \;\; \mbox{ as } n \to \infty.
  \]
  The above results~(\ref{eqn.thmcycles}), (\ref{eqn.renyi}) and~(\ref{eqn.apexforest})
  show that both $\apex^k \cf$ and $\ex(k+1)C$ have growth constant $2^k e$.

  In order to prove Theorem~\ref{thm.cycles}, we use a seemingly minor
  `redundant blocker' extension,
  Theorem~\ref{thm.redblock}, of the Erd\H{o}s-P\'{o}sa theorem.
  Call a blocker
  %{\color{blue} (or a blocker)}
  $B$ in a graph $G$ $k$-{\em redundant} if
  $B\setminus \{v\}$ is still a blocker for all but at most $k$ vertices $v \in B$.
  Thus a set $B$ of vertices in $G$ is a $k$-redundant blocker if and only
  there is a subset $S$ of $B$ of size at most $k$ such that each cycle
  in $G-S$ has at least two vertices in $B \setminus S$.
  Theorem~\ref{thm.redblock} says that if $G$ does not have $k+1$
  disjoint cycles then it has a small $k$-redundant blocker.
  Let us now take $f(k)$ as the least value that works in the
  Erd\H{o}s-P\'{o}sa theorem;
  %that is, as the maximum over all graphs
  %without $k+1$ disjoint cycles of the minimum size of a blocker.
  and recall that $f(k)$ is $\Theta(k \ln k)$.

\begin{theorem} \label{thm.redblock}
  If $G$ does not have $k+1$ disjoint cycles then it has a $k$-redundant
  blocker of size at most $f(k)+k$.
\end{theorem}

 The above results yield asymptotic properties
 of typical graphs without $k+1$ disjoint cycles.  We state three theorems.
 First we note that with high probability $k$ vertices really stand out
 -- they each have degree about $n/2$ whereas each other vertex has much smaller degree --
 and they form the only minimal blocker of sublinear size.
 We write $R_n \inu \ca$
 %\m{was $R_n \inu \ca_n$}
  to mean that the random graph $R_n$
 is sampled uniformly from the graphs in $\ca_n$.

\begin{theorem} \label{thm.standout}
  There is a constant $\delta>0$ such that the following holds.
  Let $k$ be a positive integer.  % and let $\eps>0$.
  %There exists $\delta>0$ such that statements $(i), (ii)$ and $(iii)$ below
  %hold with probability $1-e^{-\Omega(n)}$.
  For $n=1,2,\ldots$ let $R_n \inu \ex(k+1)C$,
  and let $S_n$ be the set of vertices in $R_n$
  with degree $ > n/\ln n$.  Then with probability $1-e^{-\Omega(n)}$ we have:
\begin{description}
  \item{(i)}  $|S_n|=k$ and $S_n$ is a blocker in $R_n$;
  \item{(ii)} each blocker in $R_n$ not containing $S_n$ has size
  $> \delta n$; and
  \item{(iii)} for any constant $\eps>0$, each vertex in $S_n$ has degree between
  $(\frac12 -\eps)n$ and $(\frac12 +\eps)n$.
\end{description}

\end{theorem}

  The second theorem on the random graph $R_n \inu \ex(k+1)C$
  %typical graphs without $k+1$ disjoint cycles
  concerns connectivity.
  %It will follow from Theorem~\ref{thm.cycles} and general results in~\cite{cmcd09}.
  Recall that the exponential generating function for the class $\ct$ of (labelled)
  trees is $T(z)= \sum_{n \geq 1} n^{n-2}z^n/n!$, and $T(\frac1{e})=\frac12$.
  Also note that by R\'enyi's result~(\ref{eqn.renyi}), for $R_n \inu \cf$ we
  have $\pr (R_n \mbox{ is connected}) \rightarrow e^{-\frac12}$
  as $n \to \infty$.

\begin{theorem}\label{thm.conn}
  Let $k \geq 0$ be an integer, and let
  $p_k = e^{-T(\frac{1}{2^k e})}$.  Then
  for $R_n \inu \ex(k+1)C$ we have
\begin{equation} \label{eqn.pk}
  \pr (R_n \mbox{ is connected }) \rightarrow p_k \;\; \mbox{ as }\; n \to \infty.
\end{equation}
  In particular, $p_0 = e^{-1/2} = 0.606531$ (as we already noted),
  $p_1 = 0.814600$,
  $p_2= 0.907879$,
  $p_3 = 0.953998$ and
  $p_4= 0.977005$ (to 6 decimal places).
\end{theorem}

  The third and final theorem presented here on the random graph $R_n \inu \ex(k+1)C$
  concerns the chromatic number $\chi(R_n)$ and the clique number
  $\omega(R_n)$.  It shows for example in the case $k=2$
  (concerning graphs with no three disjoint cycles) that both
  $\pr (\chi(R_n)=\omega(R_n)=3)$ and $\pr (\chi(R_n)=\omega(R_n)=4)$
  tend to $\frac12$ as $n \to \infty$.

\begin{theorem}\label{thm.chi}
  Let $k$ be a positive integer, and let the random graph $R$ be picked
  uniformly from the set of all graphs on $\{1,\ldots,k\}$.
  For each $n$ let $R_n \inu \ex (k+1)C$.
  Then for each $3 \leq i \leq j \leq k+2$, as $n \to \infty$
\[
  \pr\left( (\omega(R_n)=i) \land (\chi(R_n) = j)\right) \to  
  \pr((\omega(R)=i-2) \land (\chi(R) = j-2))
\]
  (and for other values of $i,j$ the left side tends to 0).
\end{theorem}

  The plan of the rest of the paper is as follows.
  First we prove Theorem~\ref{thm.apexforest} concerning the number of `apex
  forests'.
  %the proof rehearses some ideas useful for proving Theorem~\ref{thm.cycles};
  Next we prove the `redundant blocker' result,
  Theorem~\ref{thm.redblock}, which we then use in the proof of
  our main result, Theorem~\ref{thm.cycles}.
  After that, we prove the three theorems on the random graph $R_n \inu \ex(k+1)C$,
  namely Theorems~\ref{thm.standout}, \ref{thm.conn} and~\ref{thm.chi}.
  The last main section concerns the case $k=1$ on graphs with no
  two disjoint cycles;
  and finally we make some remarks concerning extensions of the results presented earlier.

  Initial work on this paper was done in 2008 while the first author was
  studying for an MSc %in Mathematics and the Foundations of Computer Science
  at the University of Oxford, with the second author as supervisor.
  It was mainly written while the second author was at the Mittag-Leffler Institute
  during April 2009; and the support of that Institute is gratefully acknowledged.
  %An early version of these results appeared (as joint work) in~\cite{vk08}.

%%%%%%%%%%%%%%%%%%%%%%%%%%%%%%%%%%%%%%%%%%%%%%%%%%%%%%%%%%%%%%%%%%%%%
%%%%%%%%%%%%%%%%%%%%%%%%%%%%%%%%%%%%%%%%%%%%%%%%%%%%%%%%%%%%%%%%%%%%%

\section{Counting apex forests: proof of Theorem~\ref{thm.apexforest}}
\label{sec.apex}

  The following lemma will be useful in this section and in
  Section~\ref{sec.main-proof}.
  Call a pair of adjacent vertices in a graph a {\em spike} if it consists of a leaf
  and a vertex of degree 2, which are not contained in a component of just three vertices
  forming a path.
  %the other neighbour of which \underline{is not a leaf}
  %\m{The definition of pendant appearances does not require this
  %(so lem.spike is not a completely direct application of T 5.1)}.
  Observe that distinct spikes are disjoint.
  %which is adjacent to a vertex of degree at least 2.
  %By the `appearances theorem', Theorem~5.1 of~\cite{msw05},
  %applied to the class $\cf$ of forests, we obtain:

\begin{lemma} \label{lem.spike}
  There exist constants $a>0$ and $b>0$ such that, for $n$ sufficiently large,
  the number of forests $F \in \cf_{n}$ with less than $an$ spikes is less than
  $e^{-bn} |\cf_{n}|$.
\end{lemma}
\begin{proof}
  Let $H$ be the path of 3 vertices rooted at an end vertex.  By~(\ref{eqn.renyi}) the class
  $\cf$ of forests has a growth constant, namely $e$.
  Thus we may apply the `appearances theorem'
  Theorem~5.1 of~\cite{msw06} to lower bound the number of pendant appearances of $H$ in a
  random forest; and each such appearance yields a spike.
\end{proof}

\medskip

\begin{proofof}{Theorem~\ref{thm.apexforest}}
  By~(\ref{eqn.renyi}) we have
  \begin{equation} \label{eqn.smooth}
  |\cf_n|  \sim (n)_k \ e^{k} |\cf_{n-k}|.
  \end{equation}
  Let $n > k$, let $V=\{1,\ldots,n\}$, and consider %the multiset $\ca_n'$ of all
  the following constructions of graphs on $V$:

\begin{description}
 \item{(i)} Choose a $k$-set $S \subset V$, and put any graph on $S \;$
 ($\: {n \choose k} 2^{\binom k 2}$ choices)
 %\item Put any graph on the vertices in $S$ ($2^{\binom k 2}$ choices)
 \item{(ii)} Put any forest $F$ on $V \setminus S \;$ ($\: |\cf_{n-k}|$ choices)
 \item{(iii)} Add the edges of any bipartite graph $H$ with parts $S$
 and $V \setminus S \;$
 ($\: 2^{k(n-k)}$ choices).
\end{description}
  Clearly each graph constructed is in $(\apex^k\cf)_n$, and each graph in
  $(\apex^k \cf)_n$ is constructed at least once.
  By~(\ref{eqn.smooth}) the number of constructions is
  \[
  \binom n k 2^{\binom k 2} 2^{k(n-k)} |\cf_{n-k}| \sim c_k 2^{k n} |\cf_n|
  \]
  so $|(\apex^k \cf)_n|$ is at most this number.

  Let us bound $|(\apex^k \cf)_n|$ from below by showing that almost all of the
  constructions yield distinct graphs. Observe that $G \in (\apex^k \cf)_n$ appears
  just once if and only if $G$ has a unique blocker of size $k$.
  Fix $S = S_0$ for some $k$-set $S_0 \subseteq V$.

  Let us say that a graph $G \in (\apex^k \cf)_n$ is {\em good} if
  (a) $G - S_0 \in \cf$; and (b) for each vertex $s \in S_0$
  the forest $G-S_0$ has $k+1$ spikes
  %$X_1(s), X_2(s), \dots, X_{k+1}(s) \subset V(G)\setminus S_0$
  such that $s$ is adjacent to both vertices in each of these spikes,
  and so forms a triangle with each.
  If $G$ is good then $S_0$ must be the unique blocker of size $k$ in $G$.
  For if $S'$ is another blocker,
  and $s \in S_0\setminus S'$, then $S'$ must contain a vertex from each
  spike in $G-S_0$ which forms a triangle with $s$,
  %$X_1(v), X_2(v), \dots, X_{k+1}(v)$,
  and so $|S'| \ge k+1$.

  By Lemma~\ref{lem.spike},
  there exist constants $a>0$ and $b>0$ such that
  (assuming $n$ is sufficiently large)
  the number of forests $F \in \cf_{n-k}$ with less than $an$
  spikes is less than $e^{-bn} |\cf_{n-k}|$.
  But if $F$ has at least $an$ spikes then there are at most
\[
  2^{k(n-k)} k \ \pr\left(\bin \left(\lceil an\rceil, 1/4\right) \le k \right)
\]
  ways to choose the bipartite graph $H$ in step (iii)
  %with parts $S_0$ and $V \setminus S$
  so that the resulting graph constructed is not good.
  Hence, by considering separately the cases when $F$ has $<an$ spikes
  and when $F$ has $\geq an$ spikes,
  we see that the number of ways to choose $F$ and $H$
  so that the resulting graph is constructed more than once is at most
%\begin{equation} \label{eqn.ub1}
\[
  2^{k(n-k)} |\cf_{n-k}| \left(e^{-bn}\!+
  k\ \pr\!\left(\!\bin (\lceil an \rceil,\frac14) \le k \right)\right)
  \; = \; 2^{kn} |\cf_{n}| \ e^{-\Omega(n)}
  \]
%\end{equation}
  by a standard Chernoff bound for the binomial distribution.
  Summing over all sets $S$ and all graphs on $S$, we see that the total number
  of constructions that fail to yield a unique graph also has the upper bound
  %has the same upper bound~(\ref{eqn.ub1}),
  \begin{equation} \label{eqn.ub1}
    2^{kn} |\cf_{n}| \ e^{-\Omega(n)}
  \end{equation}
%\[
% |Apex^k(\ca)_n| \ge c_k 2^{kn} |\ca_n|\left(1  - e^{-\Omega(n)} \right),
%\]
  which completes the proof.
\end{proofof}

%%%%%%%%%%%%%%%%%%%%%%%%%%%%%%%%%%%%%%%%%%%%%%%%%%%%%%%%%%%%%%%%%%%
%%%%%%%%%%%%%%%%%%%%%%%%%%%%%%%%%%%%%%%%%%%%%%%%%%%%%%%%%%%%%%%%%%%

\section{Redundant blockers: proof of Theorem~\ref{thm.redblock}}
\label{sec.redblock}

  \noindent
  We will deduce Theorem~\ref{thm.redblock} easily from the following lemma.

\begin{lemma}\label{lem.3}
  %For each integer $k\geq 0$ let $P(k)$ be the following proposition.
  Let $k \geq 0$, let $G \in \ex (k+1)C$, and let $Q$ be a blocker in $G$.
  Then there are sets $S \subseteq Q$ with $|S| \le k$ and
  $A \subseteq V(G) \setminus Q$ with $|A| \le k$ such that
  there is no cycle $C$ in $G-S$ with
  $|V(C) \cap (Q \cup A)| \leq 1$.
  %Then $P(k)$ holds for each $k \geq 0$.
\end{lemma}
  \noindent
  Note that the conclusion of the lemma is equivalent to saying that the graph
  $G-((Q \setminus \{x\})\cup A)$ is acyclic for each vertex
  $x \in Q \setminus S$; that is,
  each vertex $x \in Q \setminus S$ has at most one edge to each tree
  in the forest $G-(Q\cup A)$.
  \medskip

\begin{proof} \hspace{.05in}
  We use induction on $k$.
  Clearly the result holds for the case $k=0$, as we may take $A=S=\emptyset$.
  Let $j \ge 1$ and suppose that the result holds for $k=j-1$.
  Let $G \in \ex (j+1)C$ and let $Q$ be a blocker in $G$.
  We may assume that for some tree $T$ in the forest $G-Q$,
  and some vertex $x \in Q$, the induced subgraph
  $G[V(T)\cup\{x\}]$ has a cycle (as otherwise we may again take
  $A=S=\emptyset$ and we are done).

  Fix one such tree $T$, and fix a root vertex $r$ in $T$.  For each
  vertex $v$ in $T$ let $T_v$ denote the subtree of $T$ rooted at $v$.
  (Thus $T_r$ is $T$.) Let
\[
  R = \{v\in V(T):\ G[V(T_v)\cup\{x\}] \mbox{ has a cycle for some } x\in Q\}.
\]
  By our assumption $R \neq \emptyset$.
  In the tree $T$, choose a vertex $u \in R$ at maximum distance from
  the root $r$.
  %and note that there is no other vertex of $T_u$ in $R$.
  Let $z \in Q$ be such that $G[V(T_u)\cup\{z\}]$ has a cycle.

  Let $G' = G - (V(T_u) \cup \{z\})$ and let
  $Q' = Q \setminus \{z\}$.
  Then clearly $G' \in \ex(j C)$ and $Q'$ is a blocker in $G'$.
  Hence we can apply the induction hypothesis
  to $G'$ and $Q'$, and obtain sets of vertices $S' \subseteq Q'$
  and $A' \subseteq V(G') \setminus Q'$ each of size
  at most $j-1$, such that
  there is no cycle $C$ in $G'-S'$ with
  $|V(C) \cap (Q' \cup A')| \leq 1$.

  % each vertex $x \in Q' \setminus S'$
  %has at most one edge to each tree of the forest $G'-(Q'\cup A'))$.

  Now set $S = S' \cup \{z\}$ and $A = A' \cup \{u\}$.
  Suppose that there is a cycle $C$ in $G-S$ with
  $|V(C) \cap (Q \cup A)| \leq 1$.
  %
  %Let $x \in Q \setminus S$ and suppose that $x$ has at least two edges to
  %some tree in the forest $G-(Q\cup A)$, and so $x$ is in a cycle $C$
  %in $G$ with $V(C) \setminus \{x\} \subseteq V(G) \setminus (Q \cup A))$.
  We want to find a contradiction, since that will establish the induction step,
  and thus complete the proof of the lemma.

  Note that $C$ must have a vertex in the blocker $Q$: so we may
  let $x \in Q$ be the unique vertex in $V(C) \cap (Q \cup A)$.
  It follows that $u \not\in V(C)$.
  % since $x \in V(C)$ and $|V(C) \cap (Q \cup A)| \leq 1$.
  But $V(C) \cap V(T_u)$ cannot be empty: for then
  $C$ would be a cycle in $G'-S'$,
  %(note that $z \in S$ so $z \not\in V(C)$),
  and by induction we would have
  $2 \leq |V(C) \cap (Q' \cup A')| \leq |V(C) \cap (Q \cup A)|$.

  Hence the connected graph $C\!-\!\{x\}$
  is a subgraph of $T$ with a vertex in $T_u$ but not containing $u$.
  Therefore $C\!-\!\{x\}$ must be contained in a proper subtree $T_w$ of $T_u$;
  but this implies that $w \in R$, which contradicts our choice of $u$.
\end{proof}
  \medskip

\begin{proofof}{Theorem~\ref{thm.redblock}}
  Let $k \geq 0$ and let $G \in \ex(k+1)C$.
  Let $Q$ be a blocker in $G$ of size at most $f(k)$.
  By Lemma~\ref{lem.3}, there are sets $S \subseteq Q$ with $|S| \le k$ and
  $A \subseteq V(G) \setminus Q$ with $|A| \le k$ such that
  there is no cycle $C$ in $G-S$ with
  $|V(C) \cap (Q \cup A)| \leq 1$.

  Then the set $B=Q \cup A$ is as required.
  For, given $v \in B \setminus S$,
  there cannot be a cycle $C$ in $G-(B \setminus \{v\})$, since
  $C$ would be a cycle in $G-S$ with $|V(C) \cap B| \leq 1$;
  and thus $B \setminus \{v\}$ is a blocker.
\end{proofof}

%%%%%%%%%%%%%%%%%%%%%%%%%%%%%%%%%%%%%%%%%%%%%%%%%%%%%%%%%%%%%%%%
%%%%%%%%%%%%%%%%%%%%%%%%%%%%%%%%%%%%%%%%%%%%%%%%%%%%%%%%%%%%%%%%

\section{Proof of the main theorem, Theorem~\ref{thm.cycles}}
\label{sec.main-proof}

  If the random variable $X$ has the Poisson distribution with mean~1,
  then for each positive integer $t$ we have $\pr[X \geq t] \leq 1/t!$.  Hence
  by Theorem 2.1 of McDiarmid, Steger, Welsh (2006)~\cite{msw06}
  applied to the class $\cf$ of forests we have:

\begin{lemma}  \label{lem.msw}
  For each positive integer $t$
  \[
    |\{F \in \cf_n: \kappa(F) \geq t+1\}| \leq |\cf_n|/t! .
  \]
\end{lemma}

  The idea of the proof of Theorem~\ref{thm.cycles}
  is to give constructions which yield
  every graph in $(\ex (k\!+\!1)C)_n$ at least once
  (as well as other graphs);
  show that there are few `unrealistic' constructions;
  and show that few `realistic' constructions yield a graph
  in $(\ex (k\!+\!1)C \setminus \apex^k \cf)_n$.
  \medskip

\begin{proofof} {Theorem~\ref{thm.cycles}}
  Fix a positive integer $k$.
  By Theorem~\ref{thm.redblock}, there is an integer $r \leq f(k)+k$
  such that the following holds.
  For each graph $G$ in $\ex (k\!+\!1)C$ with at least $r$ vertices,
  there is a blocker $R$ of size $r$ and a subset $S$ of $R$
  of size $k$ such that $R \setminus v$ is still a blocker
  for each vertex $v \in R \setminus S$.

  Let $n > r$.
 Then the following constructions will yield every graph in
 $(\ex (k\!+\!1)C)_n$ at least once (as well as other graphs).
 %The restriction in (iv) arises from Theorem~\ref{thm.redblock}.
 %
\begin{description}
\item{(i)} Choose an $r$-subset $R \subseteq V$, put any graph on $R$,
  and choose a $k$-subset $S \subseteq R \; $
  ($\: {n \choose r} 2^{r \choose 2} {r \choose k} = O(n^r)$ choices)

%\item{(2)} Choose a $k$-set $S \subseteq R$ (${r \choose k}$ choices)

\item{(ii)} Add the edges of any bipartite graph %$H(S,V \setminus R)$
  with parts $S$ and $V \setminus R \;$
  ($\: 2^{k(n-r)}$ choices)

\item{(iii)} Put any forest $F$ on $V \setminus R \; $ ($\: |\cf_{n-r}|$ choices)

\item{(iv)} Add the edges of any bipartite graph %$H(R \setminus S,V \setminus R)$
  with parts $R \setminus S$ and $V \setminus R$,
  subject to the restriction that each $v \in R \setminus S$ has at most one
  edge to each component tree of the forest $F$ on $V \setminus R$.
  %for each vertex $v \in R \setminus S$,
  %the induced subgraph on $V \setminus (R \setminus \{v\})$ is a forest
  %(number of choices to be bounded).
\end{description}

  We want upper bounds on the number of constructions.
  %yielding
  %a graph in $\ex( (k\!+\!1)C_3)_n \setminus \apex^k(\cf)_n$.
  %
  By the restriction in (iv) above, for each vertex $v \in R \setminus S$,
  the number of edges between $v$ and the vertices in $V \setminus R$ is at most
  $\kappa(F)$.  Let $t=t(n) \sim n (\ln n)^{-\frac12}$.
  Then by Lemma~\ref{lem.msw}
\[
  |\{F \in \cf_{n-r}: \kappa(F) \geq t \}| \leq |\cf_{n-r}|/(t-1)!
  \leq |\cf_{n}| \ e^{-\Omega(n (\ln n)^{\frac12})}.
\]
  Call a construction {\em realistic} if there are at most
  $t$ edges between each vertex $v \in R \setminus S$ and the vertices
  in $V \setminus R$; and {\em unrealistic} otherwise.
  %for $n$ sufficiently large.
  Then the number of unrealistic constructions
  %which have more than $t$ edges between some vertex $v \in R \setminus S$
  %and the vertices in $V \setminus R$
  is at most
\[
  O(n^r) \ 2^{kn} \ |\cf_{n}| \ 2^{(r-k)(n-r)} \ e^{-\Omega(n (\ln n)^{\frac12})} =
  |\cf_{n}| \ e^{-\Omega(n (\ln n)^{\frac12})}.
\]
  %By comparing this bound with~(\ref{eqn.apexforest}) or otherwise,
  Thus there are relatively few unrealistic constructions, and
  we see that we need to consider only realistic constructions.
  % such that there are at most
  %$t$ edges between each vertex $v \in R \setminus S$ and the vertices
  %in $V \setminus R$: let us call these {\em realistic} constructions.
  %
  Further, since $t=o(n)$, for $n$ sufficiently large
  %${n-r \choose i-1} \leq \frac12 {n-r \choose i}$ for each $i=1,\ldots,t$;
  %and so
\[
  \sum_{i=0}^{t} {n-r \choose i} \leq 2 {n-r \choose t} \leq 2 \left( \frac{ne}{t} \right)^t;
\]
  and so, in realistic constructions, the number of choices for step (iv) is
\[
  \left( \sum_{i=0}^{t} {n-r \choose i} \right)^{r-k}
    \leq 2^r \left(\frac{ne}{t}\right)^{tr} = (1+o(1))^n.
\]

  Let us bound the number of realistic constructions yielding
  a graph $G$ in $(\ex (k\!+\!1)C \setminus \apex^k \cf)_n$.
  Each such construction has a cycle contained in $V \setminus S$;
  and such a cycle $C$ can touch at most $2(r-k)$ spikes,
  since as we travel around $C$ we must visit $R \setminus S$ at least
  once between any three visits to distinct spikes.

  Now suppose that each vertex in S is adjacent to both
  vertices of at least $2r-k$ spikes.
  Then the $k$ vertices in $S$ would each form triangles with
  at least $2r-k-2(r-k)=k$ spikes disjoint from $C$;
  and amongst these triangles we could find $k$ disjoint ones
  %each disjoint from $C$
  (for example by picking the triangles greedily).
  But together with $C$ there would now be at least $k+1$ disjoint cycles,
  contradicting the assumption that $G \in \ex (k\!+\!1)C$.
  Hence, for at least one vertex $v$ in S, $v$ must be adjacent to both
  vertices of at most $2r-k-1 \leq 2r$ spikes.

  Therefore, given any choices at steps (i),(iii) and (iv),
  %yielding a cycle contained in $V \setminus S$,
  if $F$ has $z$ spikes then the number of choices at step (ii) to obtain
  a graph in $(\ex (k\!+\!1)C \setminus \apex^k\cf)_n$ is at most
  \[
  2^{k(n-r)} \ k\ \pr[\bin ( z , 1/4) \le 2r].
  \]

  By Lemma~\ref{lem.spike},
  there exist constants $a>0$ and $b>0$ such that
  (assuming $n$ is sufficiently large)
  the number of graphs $F \in \cf_{n-r}$ with less than $an$ spikes
  is at most $e^{-bn} |\cf_{n-r}|$.
  Hence, by considering separately the cases when $F$ has $<an$ spikes
  and when $F$ has $\geq an$ spikes,
  we see that the number of realistic constructions which yield a graph in
  $(\ex (k\!+\!1)C \setminus \apex^k \cf)_n$ is at most
\begin{eqnarray*}
  && O(n^r) \ 2^{k(n-r)} 2^r \left(\frac{ne}{t}\right)^{tr} |\cf_{n-r}|
  \left(e^{-bn} +k\ \pr[\bin (\lceil an \rceil, 1/4) \le 2r] \right)\\
  & = & e^{-\Omega(n)} 2^{kn} \ |\cf_n| \;\; = \;\; e^{-\Omega(n)} |(\apex^k \cf)_n|
\end{eqnarray*}
  by a Chernoff bound as before, and by Theorem~\ref{thm.apexforest}.
\end{proofof}

%%%%%%%%%%%%%%%%%%%%%%%%%%%%%%%%%%%%%%%%%%%%%%%%%%%%%%%%%%%%%%%%%%%%%%%%%%%%%%%%
%%%%%%%%%%%%%%%%%%%%%%%%%%%%%%%%%%%%%%%%%%%%%%%%%%%%%%%%%%%%%%%%%%%%%%%%%%%%%%%%

\section{Proofs for random graphs $R_n$}
\label{sec.random}

  In this section we prove Theorems~\ref{thm.standout}, \ref{thm.conn}
  and~\ref{thm.chi}.
  The following lemma makes the task more straightforward.  Recall
  that the total variation distance $d_{TV}(X,Y)$ between two random
  variables $X$ and $Y$ is the supremum over all events $A$ of
  $|\pr(X \in A) - \pr(Y \in A)|$.

\begin{lemma} \label{lem.tv}
  Let $k$ be a positive integer.
  Let $R_n \in_u \ex(k+1)C$; let $R^a_n \in_u \apex^k \cf$; and
  let $R^c_n$ denote the graph which is the result of a construction
  as in the proof of Theorem~\ref{thm.apexforest},
  where the steps are chosen uniformly at random.
  If $X_n$ and $Y_n$ are any two of these random variables, then
  the total variation distance between them satisfies
  \begin{equation} \label{eqn.rndash}
    d_{TV}(X_n,Y_n) = e^{-\Omega(n)}.
  \end{equation}
\end{lemma}

\begin{proof}
  Theorem~\ref{thm.cycles} gives
  $d_{TV}(R_n,R^a_n) = e^{-\Omega(n)}$;
  and Theorem~\ref{thm.apexforest} and the inequality~(\ref{eqn.ub1}) give
   $d_{TV}(R^a_n,R^c_n) = e^{-\Omega(n)}$.
\end{proof}
\bigskip

\begin{proofof}{Theorem~\ref{thm.standout}}
  By Lemma~\ref{lem.tv}, we may work with $R^c_n$ %conditioned on $S=\{1,\ldots,k\}$,
  rather than with $R_n$.
  Let $F_m \inu \cf_m$ for $m=1,2,\ldots$.
  % let $F_m \inu \cf_m$ and let $T_m \inu \ct_m$.
  %
  If positive numbers $n_1,\ldots,n_j$ sum to at most $m$
  then $\prod_i n_i \leq \left(\frac{m}{j}\right)^j$.  Also, if vertex 1 has degree
  $j$ in $F_m$ and we delete this vertex then we obtain a forest
  with at least $j$ components.
  Thus by considering the component sizes in $F_{m-1}$, and using Lemma~\ref{lem.msw}
\begin{eqnarray*}
  \pr(\Delta(F_{m}) =j)
 & \leq &
  m \cdot \left(\frac{m}{j}\right)^j \frac{|\cf_{m-1}|}{(j-1)!}\ \frac{1}{|\cf_m|} \\
 & \leq &
  j \left( \frac{m}{j} \right)^j \frac{1}{j!}
  \;\;\; \mbox{ since } m |\cf_{m-1}| \leq |\cf_m|\\
 & \leq &
  %j \left( \frac{m}{j} \right)^j  \left( \frac{e}{j-1}
  %\right)^{j-1} \;\; \leq \;
  j \left(\frac{m e}{j^2}\right)^j \;\;\; \mbox{ since }\; j! \geq (j/e)^j.
\end{eqnarray*}
  Hence  $\pr(\Delta(F_{m}) \geq j)= e^{-\Omega(m)}$ if $j = \Omega(m/\ln m)$.

  The key observation now is that
  \[
    \pr(S_n \not\subseteq S) \leq \pr(\Delta(F_{n-k}) >  n/\ln n -k),
  \]
  and so by the above $\pr(S_n \not\subseteq S) = e^{-\Omega(n)}$.
  %
 % The last step follows from the fact that, for positive integers $m$ and $t$,
 % \[
 % \pr(\bin(m,1/m) \geq t) \leq {m \choose t} m^{-t} \leq (e/t)^t.
 % \]
  But the number of constructions with $S_n$ a proper subset of $S$ is at most
  $2^{(k-1)n + o(n)} |\cf_{n-k}|$, which is $2^{-n+o(n)}$ times the number of constructions,
  and hence
  $\pr(S_n=S) = 1-  e^{-\Omega(n)}$.

  We have now dealt with statement (i) in the theorem, so let us consider statement (ii).
  By Lemma~\ref{lem.spike}, there exists $\delta>0$ such that
  the probability that $F_{n-k}$ has a matching of size at least $5 \delta n$ is
  $ 1-  e^{-\Omega(n)}$;
  and given such a matching, for each $j \in S$, the probability that $j$ fails to be
  the central vertex of at least $\delta n$ otherwise disjoint triangles is at most
  \[
    \pr\left(\bin( \lceil 5 \delta n \rceil, \frac14) < \delta n \right) = e^{-\Omega(n)}
  \]
  by a Chernoff bound.
  But if vertex $j$ is the central vertex of at least $\delta n$
  otherwise disjoint triangles, then any blocker not containing $j$ must
  have size at least $\delta n$.  This deals with statement (ii).

  Finally, for statement (iii), a Chernoff bound
  %standard inequalities for the binomial distribution
  shows that the number of constructions such that (iii) fails is
  $2^{kn -\Omega(n)} |\cf_{n-k}|$;
  and it follows that (iii) holds with probability $1-e^{-\Omega(n)}$.
\end{proofof}
\bigskip

  In order to prove Theorem~\ref{thm.conn} we shall use Lemma 4.3 of~\cite{cmcd09}.
  We need some definitions to present that lemma (in a simplified form).

  Given a graph $G$ on $\{1,\ldots,n\}$
  let $\Bigc(G)$ denote the (lexicographically first) component of $G$
  with the most vertices, and let $\Frag(G)$ denote the graph induced on
  the vertices not in $\Bigc(G)$.  Let $\ca$ be a class of graphs.
  We say that $\ca$ is {\em bridge-addable} if given any graph in $\ca$
  and vertices $u$ and $v$ in distinct components of $G$,
  the graph obtained from $G$ by adding an edge joining $u$ and $v$ must be in $\ca$.
  Given a graph $H$ in $\ca$, we say that $H$ is {\em freely addable}
  to $\ca$ if, given any graph $G$ disjoint from $H$,
  the union of $G$ and $H$ is in $\ca$ if and only if $G$ is in $\ca$.
  We say that the class $\ca$ is {\em smooth}
  if $\ca$ has growth constant $\gamma$ and
  $\frac{|\ca_n|}{n |\ca_{n-1}|} \to \gamma$ as $n \to \infty$.
  Finally, note our standard convention that for the class $\ca$
  we will use $A(z)$ to denote its exponential generating function
  $\sum_{n \geq 0}|\ca_n| z^n/n!$.
  \begin{lemma} {(McDiarmid~\cite{cmcd09})}
  \label{lem.gen}
  Let the graph class $\ca$ be bridge-addable;
  let $R_n \inu \ca_n$;
  let $\cb$ denote the class of all graphs freely addable to~$\ca$;
  and suppose that
  $\pr(\Frag(R_n) \in \cb) \to 1$ as $n \to \infty$.
  %and let $\cd$ denote the class of connected graphs $\cb$.
  Suppose further that $\ca$ is smooth,
  with growth constant $\gamma$.
  %and let $\rho = 1/\gamma$.
  Let $\cc$ denote the class of connected graphs $\cb$.
  Then $C(1/\gamma)$ is finite, and
  %and the number $\kappa(R_n)$ of components of $R_n$
  %converges in distribution to $1+ \Po(D(\rho))$,
  %where $\Po(\lambda)$ denotes the Poisson distribution with mean $\lambda$.
  %In particular
  \[
   \pr[R_n\mbox{ is connected }] \;\to\;  e^{-C(1/\gamma)} \;\;
   \mbox{ as } n \to \infty.
  \]
\end{lemma}

\begin{proofof}{Theorem~\ref{thm.conn}}
  By Lemma~\ref{lem.tv}, we may work with $R^a_n$, rather than with $R_n$.
  Let $\ca$ denote $\apex^k \cf$: thus $R^a_n \inu \ca_n$.
  Clearly $\ca$ is bridge-addable,
  and the class of graphs freely addable to $\ca$ is $\cf$.
  By Theorem~\ref{thm.apexforest}, $\ca$ is smooth, with growth constant $2^k e$.
  By Lemma~\ref{lem.gen} above, it now remains only to show that
  $\pr(\Frag(R^a_n) \in \cf) \to 1$ as $n \to \infty$.
  We may assume that $k \geq 1$.

  By Theorem~\ref{thm.apexforest}, the class $\apex^{k-1} \cf$
  has growth constant $2^{k-1}e$,
  and so the class $\cd$ of graphs with each component in $\apex^{k-1} \cf$
  also has growth constant $2^{k-1}e$ (by the `exponential formula').
  If $G \in (\ca \setminus \cd)_n$ and $\Frag(G) \not\in \cf$,
  then apart from a unique component of size at most
  $\lfloor n/2 \rfloor$ which is in $\ca \setminus \apex^{k-1} \cf$
  %$\apex^k(\cf) \setminus \apex^{k-1}(\cf)$,
  the rest of the graph is in $\cf$;
  and the number of such graphs is at most
\[
  \sum_{t=0}^{\lfloor n/2\rfloor} {n \choose t} |\ca_t| \cdot | \cf_{n-t}|
  = n! (e+o(1))^n 2^{kn/2} =  2^{-kn/2 +o(n)} \cdot |\ca_n|.
\]
  Thus $\pr(\Frag(R^a_n) \not\in \cf) = e^{-\Omega(n)} = o(1)$.
\end{proofof}
\medskip

\noindent
  Lemmas 4.3 and 4.4 in~\cite{cmcd09} may be used to yield further results
  on $\Frag(R_n)$.
\medskip

\begin{proofof}{Theorem~\ref{thm.chi}}
  By Lemma~\ref{lem.tv} it is sufficient to consider $R^c_n$ rather than $R_n$.
  But it is easy to see that, with probability $\to 1$ as $n \to \infty$,
  there are adjacent vertices in $V \setminus S$ which are adjacent to each vertex
  in $S$; and the theorem follows.
\end{proofof}

%%%%%%%%%%%%%%%%%%%%%%%%%%%%%%%%%%%%%%%%%%%%%%%%%%%%%%%%%%%%%%
%%%%%%%%%%%%%%%%%%%%%%%%%%%%%%%%%%%%%%%%%%%%%%%%%%%%%%%%%%%%%%

\section{No two disjoint cycles}
\label{sec.no2}

  %For each positive integer $k$,
  Let $\cd^k$ denote the `difference' class
  $\ex (k\!+\!1)C \setminus \apex^k \cf$,
  the class of graphs with no $k+1$ disjoint cycles but with
  no blocker of size at most $k$.
  Our main result, Theorem~\ref{thm.cycles},
  shows that $\cd^k$ is exponentially smaller than $\ex (k\!+\!1)C$.
  For the case $k=1$ we can say much more about $\cd = \cd^1$, based on results from 1965 of
  Dirac~\cite{dirac65} and Lov\'asz~\cite{lovasz65}, see also Lov\'asz~\cite{lovasz93} problem 10.4.

  We need some definitions and notation.
  The {\em 2-core} or just {\em core} of a graph $G$ is the unique maximal subgraph of
  minimum degree at least 2, and is denoted by $\core(G)$.  %(It may be empty.)
  Let $\tilde{\ck}$ denote the class of graphs homeomorphic to $K_5$;
  let $\tilde{\cb}$ denote the class of graphs homeomorphic to a multigraph
  $\tilde{K}_{3,t}$ formed from the complete bipartite graph $K_{3,t}$ for some $t \geq 0$
  %\m{was $t \geq 3$}
  by possibly adding edges or multiple edges between vertices in the `left part' of size 3
  ($K_{3,0}$ has only a `left part');
  and let $\tilde{\cw}$ denote the class of graphs homeomorphic to a multigraph %$\tilde{W}_t$
  formed from the $t$-vertex wheel $W_t$ for some $t \geq 4$ by possibly adding parallel
  edges to some spokes.
  Let $\ck$, $\cb$, $\cw$ denote the classes of graphs $G$ such that $\core(G)$ is in
  $\tilde{\ck}$, $\tilde{\cb}$, $\tilde{\cw}$ respectively.
  Call the graphs in $\cw$ {\em generalised wheels}, and note that $\cw \subseteq \cd$.

\begin{theorem} (Dirac~\cite{dirac65}, Lov\'asz~\cite{lovasz65})
  \label{thm.dl}
  \[
  \ex \ 2C = (\apex \, \cf) \cup \cw \cup \cb \cup \ck.
  \]
\end{theorem}
  By Theorems~\ref{thm.cycles} and~\ref{thm.apexforest},
  $\ex\ 2C$ and $\apex \cf$ both have growth constant $2e$,
  %and by Theorem~\ref{thm.cycles}
  %the egf $D(z)$ for $\cd$ has radius of convergence $> 1/2e$, %\approx 0.184$.
  and $\cd = \ex\ 2C \setminus \apex \cf $ is exponentially smaller.
  % than these classes.
  %$\ex(2C)$ and $\apex(\cf)$.
  The next result shows that $\cd$ is dominated by the class $\cw$ of generalised wheels,
  and gives an asymptotic formula for $|\cd_n|$.
  % and the classes $\ck$ and $\cb$ are exponentially smaller, with growth constant $e$.

\begin{theorem} \label{thm.w}
  The classes $\ck$ and $\cb$ each have growth constant $e$,
  and $\cw$ has growth constant $\gamma$ satisfying $\; e < \gamma < 2e$.  Indeed
  $|\cw_n| \sim c/n \ \gamma^n \ n!$, where the constants $c$ and $\gamma$ are given by
  equations~(\ref{eqn.rad}) and~(\ref{eqn.wplus}).
  %and $c=..$, $\gamma = 4.346..$ to 6 decimal places.
  %
  Thus $|\cd_n| \sim c/n \ \gamma^n \ n!$ so that $\cd$ has growth constant $\gamma$,
  and $\cd \setminus \cw$ has growth constant $e$.
  To 3 decimal places we have $c=0.158$ and $\gamma = 4.346$.
  %The radii of convergence satisfy
  %\[ \rho_{\ck}= \rho_{\cb} = 1/e \ ; \;\mbox{ and } \; \rho_{\cd}= \rho_{\cw} = r_2 \]
  %where
  %\m{$r_2 = 0.23008896$ to 8 dp}
  %$r_2$ is specified in~(\ref{eqn.rad1}) and~(\ref{eqn.rad2}) below,
  %and $r_2 = 0.230089$ to six decimal places.
\end{theorem}
  %
  %Thus the class $\cd$ is dominated by the class $\cw$ of generalised wheels.
  %$\cd$ is dominated by $\cw$, the graphs arising from wheels.
  %(Also, $\cd \setminus \cw$ is dominated by $\cb$, leaving just the small class $\ck$.)
  %The rest of this section is devoted to proving this result.
  \medskip

  \begin{proof}\hspace{.05in}
  Direct estimation shows easily that $\tilde{\ck}$ has growth constant 1.
  Let ${\cal R}$ denote the class of rooted trees, so that $R(z) = \sum_{n \geq 1} n^{n-1} z^n/n!$.
  %(This is our standard convention, so that $\ck$ has exponential generating function
  %$K(z)$, and so on.)
  It is well known that the radius of convergence $\rho_R$ of $R$ equals $1/e$ and $R(1/e)=1$.
  %(see for example ??).  \m{ref}
  Since graphs in $\ck$ are obtained from graphs in $\tilde{\ck}$
  by substituting rooted trees for vertices, we have $K(z)= \tilde{K}(R(z))$, and it follows that
  $\ck$ has growth constant $e$.  In a similar way we may see that $\cb$ also has
  growth constant~$e$.

  Now let us consider $\cw$.  We need to see how graphs in $\cw$ are formed from simpler graphs.
  A `hairy cycle' is a graph formed by attaching paths to a cycle.
  More precisely, a connected graph is a {\em hairy cycle} if its core is a cycle
  and each vertex not on the cycle has degree 1 or 2.
  A {\em coloured hairy cycle} is a hairy cycle in which each vertex on the cycle is coloured
  black or white.
  Let $\ch^+$ be the class of coloured hairy cycles,
  and let $\ch$ be the class of graphs in $\ch^+$ such that at least 3 vertices on the cycle
  are either coloured black or have degree at least three.
  We shall see later that the difference between $\ch^+$ and $\ch$ is negligible.

  Let $\cs$ denote the class of homeomorphs of a star (sometimes called `spiders'),
  rooted at the centre vertex,
  with the root coloured black or white. Thus the graphs in $\cs$ correspond to a
  black or white root vertex and a set of oriented paths; and so
  $S(z) = 2z e^{z/(1-z)}$.
  Recall that the exponential generating function for cycles is
  $C(z)= -\frac12 \ln(1-z) - \frac12 z - \frac14 z^2$.
  Graphs in $\ch^+$ are obtained from cycles by substituting a rooted graph from $\cs$ for each vertex,
  so $H^+(z)= C(S(z))$.

  %We need to consider graph classes which are slightly larger than $\tilde{\cw}$ and $\cw$:
  %we shall see later that the differences are negligible.
  %
  Let $\tilde{\cw}^+$ be the class of graphs $G$ obtained by starting with a root vertex $v$
  and a graph $H \in \ch^+$ not containing $v$; and joining $v$ to each leaf of $H$ and to
  each black vertex on the cycle in $H$, and then removing all colours.
  If the initial graph $H$ is in $\ch$ then $G \in \tilde{\cw}$
  (the rooting of $v$ is irrelevant since the `centre' vertex of a wheel is unique,
  so we may say $\tilde{\cw} \subseteq \tilde{\cw}^+$).
  Conversely, given a graph $G$ in $\tilde{\cw}^+$, with root vertex $v$,
  colour the vertices on the rim black if they are adjacent to $v$ and white otherwise,
  and then delete $v$.  We obtain a graph $H$ in $\ch^+$, and if the initial graph $G$ is in
  $\tilde{\cw}$ then $H \in \ch$.
   %(a coloured hairy cycle).
  %Conversely, from a root vertex $v$ and a graph in $\ch$,
  %we obtain a graph in $\cw$ by joining $v$ to each leaf and to each black vertex on the cycle,
  %and removing all colours.
  Hence $\tilde{W}(z) = z H(z)$ and $\tilde{W}^+(z) = z H^+(z)$.

  Let $\cw^+$ be the class of graphs formed by starting with a graph in $\tilde{\cw}^+$
  and substituting rooted trees for vertices.  (Thus $\cw^+$ is the class of graphs with 2-core in
  $\tilde{\cw}^+$, except that we always treat the root as having degree at least 2.)  Then,
  $\cw \subseteq \cw^+$, and arguing as earlier,
  $W(z)= \tilde{W}(R(z))$, and $W^+(z)= \tilde{W}^+(R(z))= R(z)\ C(f(z))$ where $f(z)=S(R(z))$.

  Observe that $S(\frac12)=e>1$, so there exists $x$ with $0<x<\frac12$ such that $S(x)=1$.
  Since $\rho_R = 1/e$ and $R(1/e)=1$, there exists $r$ with $0< r < 1/e$ such that
  $R(r)=x$; and so
  \begin{equation} \label{eqn.rad}
  f(r) = S(R(r))= 1.
  \end{equation}
  We have a supercritical composition $C(f(z))$ (see~\cite{fs09} VI.9 page 411).
  It follows from standard results as in~\cite{fs09} that
  %\m{be more explicit?}
 \begin{equation} \label{eqn.wplus}
   |\cw^+_n| \sim c/n \ \gamma^{n} \ n! \;\;\;
   \mbox{ where }\; \gamma=1/r \; \mbox{ and }\; c= \frac12 R(r).
 \end{equation}
  Finally, it is easy to see that $\ch^+ \setminus \ch$ has growth constant 1, and so
  $\cw^+ \setminus \cw$ has growth constant $e < \gamma$.  Thus the asymptotic formula
  in~(\ref{eqn.wplus}) applies also to $\cw$.

  Numerical calculations %\m{check!}
  yield $c$ and $\gamma$ as given in the theorem.
  Indeed $S(x)=1$ for $x= 0.315411$ (to six decimal places); $c=x/2$; and
  $R(r)= x$ for $r = 0.230089$ (to six decimal places).
  %To see this, observe that for $0<r<1/e$,
  %since $R(1/e)=1$ we have
  %$\sum_{m \geq n} \frac{m^{m-1}r^m}{m!} \leq (re)^n$.
  %\m{$x = 0.31541121$ to 8 dp}
  %\m{$r = 0.23008896$ to 8 dp}
  %\m{$\gamma = 4.3461451$ to 7dp?}
\end{proof}

%%%%%%%%%%%%%%%%%%%%%%%%%%%%%%%%%%%%%%%%%%%%%%%%%%%%%%%%%%%
%%%%%%%%%%%%%%%%%%%%%%%%%%%%%%%%%%%%%%%%%%%%%%%%%%%%%%%%%%%

\section{Concluding Remarks}
\label{sec.concl}

%\m{even cycles?  directed cycles?  more refs?}

  %Theorems~\ref{thm.cycles} and~\ref{thm.apexforest}
  Our results are stated for a
  fixed number $k$ of disjoint cycles, but they hold also when $k$
  is allowed to grow with $n$.
  Indeed it is straightforward to adapt the proofs to show that
  Theorem~\ref{thm.apexforest} holds as long as $k=o(n)$,
  and Theorem~\ref{thm.cycles} holds for $k = o(\ln n / (\ln\ln n)^2)$
  (in the proof take $t= \omega(n)\ n/\ln n$ where $\omega(n) \to \infty$
  slowly as $n \to \infty$).

  It would be interesting to know more about the difference class
  $\cd^k = \ex (k\!+\!1)C \setminus \apex^k \cf$ for $k \geq 2$, ideally along the lines
  of the results on $\cd^1$ in the last section.
  %\m{radius of conv at most $r/2^{k-1}$.}
  %It would be interesting also to know to what extent there are results for
  %unlabelled graphs corresponding to the results given here for labelled graphs.
  There are results for unlabelled graphs corresponding to the results given here
  for labelled graphs -- see~\cite{kmcd09}.
  
  The Erd\H{o}s-P\'{o}sa theorem was extended from disjoint cycles to
  suitable more general disjoint graph minors by Robertson and
  Seymour~\cite{rs86} in 1986.
  Our results can be extended in this direction,
  and we do so in~\cite{vk-cmcd09b}.
  For example, there is a result corresponding to Theorem~\ref{thm.cycles}
  for `long' cycles.  Fix an integer $j \geq 3$, and call a cycle {\em long}
  if it has length at least $j$.
  Then amongst all graphs $G$ on $\{ 1,\ldots,n\}$
  which do not have $k+1$ disjoint long cycles,
  all but an exponentially small proportion have a set $B$ of $k$ vertices
  such that $G-B$ has no long cycles.
  There is also a version of the Erd\H{o}s-P\'{o}sa theorem for directed
  graphs~\cite{rrst96}: what can be said in this case?

  As well as concerning a problem which is interesting in its own right,
  the results presented here are a step towards understanding the behaviour of
  random graphs from a minor-closed class
  %which is not `addable'; that is,
  where the excluded minors are not 2-connected, see the last section
  of~\cite{cmcd09}.

%%%%%%%%%%%%%%%%%%%%%%%%%%%%%%%%%%%%%%%%%%%%%%%%%%%%%%%%%%%
%%%%%%%%%%%%%%%%%%%%%%%%%%%%%%%%%%%%%%%%%%%%%%%%%%%%%%%%%%%

  {%\small
  
}

\end{document}